\documentclass{tMAM2e}

\usepackage{amsmath,amssymb,amsthm}
\usepackage{mathtools}
\usepackage{tikz-cd}
\usepackage{tikz}
\usepackage{hyperref}
\usepackage[utf8]{inputenc}
\usepackage[T1]{fontenc}

\theoremstyle{plain}
\newtheorem{theorem}{Theorem}[section]
\newtheorem{proposition}[theorem]{Proposition}

\theoremstyle{definition}
\newtheorem{definition}[theorem]{Definition}
\newtheorem{example}[theorem]{Example}

\theoremstyle{remark}
\newtheorem{remark}[theorem]{Remark}

\providecommand\sslash{\mathbin{/\mkern-5.5mu/}}

\begin{document}

\articletype{RESEARCH NOTE}
\title{Blues for Alice: The Interplay of Neo-Riemannian and Cadential Viewpoints}
%\subtitle{Tetradic Extensions via the PLRQ Group and Slice Categories}

\author{Octavio A. Agustín-Aquino}
%\address{Universidad Tecnológica de la Mixteca}
%\email{octavioalberto@mixteco.utm.mx}

%\title[Blues for Alice]{Blues for Alice: The Interplay of Neo-Riemannian and Cadential Viewpoints}
\date{\today}

%\subjclass[2020]{00A65, 18B99}

\maketitle

\begin{abstract}
We extend a property of Mazzola's theory of cadential sets in relation to the modulation between minor and major tonalities from triadic to tetradic harmony, using the PLRQ group of Cannas et al.\ (2017) as the analogue of the classical PLR group. While the PLR group connects triadic cadential sets via the relative morphism $R$, the tetradic case reveals a richer structure: two pairs of cadential sets connected by distinct morphisms forming a ``prism'' in the coslice category over the common chords of the paired cadential sets, and a single pair for those that allow quantized modulations. We demonstrate this structure through an analysis of Charlie Parker's \textit{Blues for Alice} (1951) and Ray Noble's \textit{Cherokee} (1938), showing how the prism morphism, PLRQ transformations and quantized modulations organize harmonic navigation in bebop. The categorical framework models a formal correlate of the transformational \textit{vécu} that musicians experience when navigating between cadential regions.
\end{abstract}

\begin{keywords}
Cadential sets, quantized modulations, PLRQ group, neo-Riemannian theory, seventh chords, jazz harmony, coslice category, Charlie Parker
\end{keywords}

\begin{classcode}
MSC20: 00A65, 18B99
\end{classcode}

\section{Introduction}

The problem of modulation between relative major and minor keys (such as C major and A minor) has occupied theorists since
long ago. For example, Kirnberger \citep[p. 108]{Kirnberger1774} suggests a progression $\mathrm{C}\to \mathrm{D} \to \mathrm{G} \to \mathrm{E}7 \to \mathrm{Am}$, while Rimsky-Korsakov \citep[p. 44]{RimskyKorsakov1910} uses $\mathrm{C} \to \mathrm{Dm}\to \mathrm{E}7 \to \mathrm{Am}$ to effect this modulation in a rather straightforward fashion. The principle seems to be rooted in the notion of standard cadence from the fifth degree towards the tonic.

Guerino Mazzola's theory of global tonality \citep{Mazzola2002}, on the other hand, introduced \emph{minimal cadential sets} or \emph{cadences}:  minimal collections of chords that uniquely identify a scale (or, more precisely, an interpretation of a scale), as the mechanism for asserting tonality. In order to define them, first let $E=\{x_{0},\ldots,x_{6}\}\subseteq \mathbb{Z}_{12}$ be a major scale, ordered by the circular order induced by octave equivalence. Its triadic and tetradic \emph{interpretations} are obtained by stacking generic thirds, namely:
\[
E^{(3)}=\{\{x_{i},x_{i+2\bmod{7}},x_{i+4\bmod{7}}\}:i\in\mathbb{Z}_{7}\},
\]
and
\[
E^{(4)}=\{\{x_{i},x_{i+2\bmod{7}},x_{i+4\bmod{7}},x_{i+6\bmod{7}}\}:i\in\mathbb{Z}_{7}\},
\]
respectively. The Roman numerals I, II, $\ldots$, VII denote\footnote{Since the scale is fixed in advance and chords are obtained by stacking thirds over a given scale degree, the Roman numeral denotes only the degree. The intervallic content of the resulting chord therefore depends on both the scale and on the number of stacked thirds, which makes this notation more natural for tetrads and higher structures. When useful for orientation, we indicate the corresponding traditional chord names in the text. We will also write specific degrees of the tonality $T$ as a subscript, e.g., $\text{I}_{T}$.} the starting scale degrees $x_{i}$.

For the major scale in the triadic interpretation, the cadences are:
\[
k_1 = \{\mathrm{II}, \mathrm{V}\}, k_2 = \{\mathrm{II}, \mathrm{III}\}, k_3 = \{\mathrm{III}, \mathrm{IV}\}, k_4 = \{\mathrm{IV}, \mathrm{V}\}, k_5 = \{\mathrm{VII}\}.
\]

It is important to remark that, throughout this paper, the term \emph{cadence} is used in Mazzola's technical sense of a cadential set: a minimal chordal signature that identifies an $n$-adic interpretation, i.e., the set of diatonic $n$-ads over a scale. It is not used in the common-practice sense of an ordered closing syntagm. In any case, the present analysis is not incompatible with syntactic jazz-harmonic analyses like those of ii-V motion, blues form, or turnarounds. Hence, once all the members of a cadence in Mazzola's sense have been displayed, the corresponding $n$-adic interpretation (i.e. the tonality in Mazzola's terminology) is identified. This does not mean that centricity or formal closure has been established in the common-practice syntagmatic sense.

Mazzola developed a so-called \emph{quantum theory of modulation} for the transit between one tonality and another, but it does not properly model the modulation between minor and major tonalities, particularly between a major tonality and its relative minor (for they are constructed over the same scale) thus rendering them as non-quantized. Also del Pozo and Gómez-Martin propose a modulation theory based on discrete optimization \citep{dPGM22}, but it is trivial for a major tonality and its relative minor, since their scales coincide.

In order to manage this modulation we resort to the PLR group\footnote{See \citep{Crans01062009} for a good introduction on this group.}, acting on the set of major and minor triads, for it provides morphisms connecting these cadential sets. In particular, the relative $R$ satisfies $R \circ T^5 = T^5 \circ R$ and $R \circ T^7 = T^7 \circ R$, yielding the commutative diagram
\begin{equation}\label{E:RforTriads}
\begin{tikzcd}
\mathrm{F} & \mathrm{C} \arrow[l,swap,"T^{5}"] \arrow[r,"T^{7}"] & \mathrm{G} \\
\mathrm{Dm} \arrow[<->,u, "R"] & \mathrm{Am} \arrow[l, "T^{5}"] \arrow[r,swap,"T^{7}"] \arrow[<->,u, "R"] & \mathrm{Em} \arrow[<->,u, "R"]
\end{tikzcd}
\end{equation}

Thus $R$ sends the cadence $k_4 = \{\mathrm{IV}, \mathrm{V}\}$ to $k_{2} = \{\mathrm{II}, \mathrm{III}\}$ (which lies within the relative minor), which is similar to Rimsky-Korsakov's solution. The relative also transforms $k_{1} = \{\mathrm{II}, \mathrm{V}\}$ into $k_{3} = \{\mathrm{III}, \mathrm{IV}\}$ and back, but only as sets, for there is no commutative diagram that clarifies the relation.
		
This paper extends this framework to \emph{tetradic} (seventh chord) harmony, where the relevant symmetry group is PLRQ \citep{Cannas2017}, and discovers a conglomerate structure connecting six cadential sets.

\section{Tetradic Cadential Sets}

Regarding the tetradic interpretation of the major scale, the cadential sets are \citep{AgustinAquino2020}:
\begin{multline*}
J_1 = \{\mathrm{I}^7, \mathrm{II}^7\},
J_2 = \{\mathrm{I}^7, \mathrm{IV}^7\},\\
J_3 = \{\mathrm{II}^7, \mathrm{III}^7\},
J_4 = \{\mathrm{III}^7, \mathrm{IV}^7\},\\
J_5 = \{\mathrm{V}^7\}, 
J_6 = \{\mathrm{VII}^7\}.
\end{multline*}

Note that we include the superscript to indicate that these belong to the tetradic interpretation of the major scale.

\begin{remark}
A fundamental difference from the triadic case is that the cadential sets $J_1$ and $J_2$ contain the tonic, and that there exist two singletons: $J_5$ and $J_6$. This asymmetry will structure our categorical analysis.
\end{remark}

\section{The PLRQ Group}

The PLRQ group \citep{Cannas2017} generalizes PLR to seventh chords. We will enlist the morphisms that we require from it, each with an illustrative example.

\begin{definition}[$R_{42}$]
The morphism $R_{42}$ exchanges major seventh (represented by the subscript $4$) and minor seventh chords (represented by the subscript $2$) whose roots are related by minor third:
\[
R_{42}: [\underline{x}, x+4, x+7, x+11] \leftrightarrow [x, x+4, x+7, \underline{x+9}].
\]
where the underline indicates the root.
\end{definition}

\begin{example}
With $x = 5$, in the scale of $C$:
\[
\mathrm{IV}^7 = [5,9,0,4] \mapsto [5,9,0,2] = [2,5,9,0] = \mathrm{II}^7.
\]
\end{example}

\begin{definition}[$L_{13}$]
The morphism $L_{13}$ exchanges dominant seventh (represented by the subscript $1$) and half-diminished seventh (represented by the subscript $3$) chords whose roots are related by a major third:
\[
L_{13}: [\underline{x}, x+4, x+7, x+10] \leftrightarrow [x+2, \underline{x+4}, x+7, x+10].
\]
\end{definition}

\begin{example}
With $x = 7$, in the scale of $C$:
\[
\mathrm{V}^7 = [7,11,2,5] \mapsto [9,11,2,5] = \mathrm{VII}^7.
\]
\end{example}

\begin{definition}[$P_{42}$]
The parallel transformation $P_{42}$ converts major seventh to minor seventh chords with the same root:
\[
P_{42}: [\underline{x}, x+4, x+7, x+11] \mapsto [\underline{x}, x+3, x+7, x+10].
\]
\end{definition}

Define
\[
 L_{42}: [\underline{x},x+4,x+7,x+11]\leftrightarrow [x+2,\underline{x+4},x+7,x+11].
\]

Then, in Cannas' notation we have
\[
 T^{8}\circ L_{42} = P_{42},
\]
for
\begin{align*}
T^{8}\circ L_{42}([\underline{x},x+4,x+7,x+11])&= T^{8}([\underline{x+4},x+7,x+11,x+2])\\
&=[\underline{x},x+3,x+7,x+10],
\end{align*}
as claimed.

\begin{example}
With $x=11$:
\[
 \mathrm{Bmaj}7 = [11,3,6,10] \mapsto [11,2,6,9] = \mathrm{Bm}7.
\]
\end{example}

\begin{remark}Notice that $P_{42}$ contains two familiar triadic transformations on the overlapping triads of a seventh chord. On the lower triad it restricts to the parallel transformation $P$, while on the upper triad it restricts to the Slide transformation S. For example, Bmaj7$\mapsto$ Bm7 sends B major to B minor in the lower triad, and sends E$\flat$ minor to D major in the upper triad.
\end{remark}

\section{The Coslice Category Construction}

We need to define the underlying category for the main constructions of
the paper. Let \(\mathcal X\) be the set of dominant, minor, half-diminished,
major, and diminished seventh chords in \(\mathbb Z/12\mathbb Z\). Let
$PLRQ$ be the PLRQ group of transformations of seventh
chords as described by \cite{Cannas2017}; its generators exchange specified
pairs of chord types and fix the remaining types. We set
\[
G=\langle PLRQ,T^{s}\mid s\in\mathbb Z_{12}\rangle,
\]
where \(T^s\) denotes transposition by \(s\). We take \(\mathcal C\) to be
the action groupoid \(\mathcal X\sslash G\). More explicitly, the objects of \(\mathcal C\)
are the chords in \(\mathcal X\), and a morphism \(X\to Y\) is a pair
\((g,X)\), where \(g\in G\) and \(g(X)=Y\). Composition is given by
\[
(h,g(X))\circ(g,X)=(hg,X)
\]
and, when no confusion is possible, we will write the morphism \((g,X)\) simply as
\[
g:X\to g(X).
\]

Thus transformations such as \(R_{42}\), \(L_{13}\), \(P_{42}\), and the
relevant transpositions determine morphisms between chord objects.

The pairs of cadential sets of interest to us have common members:
\(J_1\) and \(J_2\) share \(I^7\), while \(J_3\) and \(J_4\)
share \(III^7\). Hence, if we were to generalize \eqref{E:RforTriads}
to tetrads, the difficulty is that the transformations connecting
the non-common members do not generally stabilize these common members.
For instance, \(R_{42}\) sends \(IV^7\) to \(II^7\), but it does not fix
\(I^7\). Thus the relation between paired cadential sets is not best
described as a transformation of two-element sets fixing their
intersection. We resolve this using the \emph{coslice category} construction.

\begin{definition}
Given morphisms $f: x \to y$ and $g: x \to z$ in a category $\mathcal{C}$, the \emph{coslice} category $x\downarrow \mathcal{C}$ has as objects all morphisms with domain $x$, and a morphism from $f$ to $g$ is given by $h: y \to z$ such that $h \circ f = g$.
\end{definition}

Taking the common member of each paired cadential set as the distinguished object, we ask: do morphisms exist between the various seventh chords appearing in cadential sets?

\begin{proposition}
The following morphisms connect chords in the cadential sets:
\begin{enumerate}
\item $T^5(\mathrm{I}^7) = \mathrm{IV}^7$.
\item $T^{10}(\mathrm{III}^7) = \mathrm{II}^7$.
\item $R_{42}(\mathrm{IV}^7) = \mathrm{II}^7$ and $R_{42}(\mathrm{II}^{7})=\mathrm{IV}^{7}$.
\end{enumerate}
\label{P:Triangle}
\end{proposition}

This yields commutative triangles in $\mathcal{X} \sslash G$ representing morphisms in the relevant coslice categories:
\[
\begin{tikzcd}
\mathrm{I}^7 \arrow[r, "T^5"] \arrow[dr, "R_{42}\circ T^5"'] & \mathrm{IV}^7 \arrow[d, "R_{42}"] \\
& \mathrm{II}^7
\end{tikzcd}
\qquad
\begin{tikzcd}
\mathrm{III}^7 \arrow[r, "T^{10}"] \arrow[dr, "R_{42}\circ T^{10}"'] & \mathrm{II}^7 \arrow[d, "R_{42}"] \\
& \mathrm{IV}^7
\end{tikzcd}
\]

In these coslice triangles, the object under which the coslice is taken is the common member of a pair of cadential sets.
Thus the triangle with distinguished vertex I$^{7}$ relates $J_{1}=\{\text{I}^{7},\text{II}^{7}\}$ and $J_{2}=\{\text{I}^{7},\text{IV}^{7}\}$. The outgoing arrows have the non-common members II$^{7}$ and IV$^{7}$ as codomains and $R_{42}$ connects them.
Similarly, the triangle with distinguished vertex III$^{7}$ relates $J_{3}=\{\text{III}^{7},\text{II}^{7}\}$ and $J_{4}=\{\text{III}^{7},\text{IV}^{7}\}$. Thus the morphisms in Proposition \ref{P:Triangle} make explicit the relations between pairs of cadences
by representing each cadence through the corresponding object of the appropriate coslice category.

\section{The ABC Conglomerate Structure}

\begin{theorem} There exists a commutative diagram that relates $(J_{1},J_{2})$ and $(J_{3},J_{4})$, as follows
\begin{equation}\label{E:Prism}
    \begin{tikzcd}[row sep=1.5em, column sep = 1.5em]
    \mathrm{IV}^{7} \arrow[rr,"R_{42}"] \arrow[dr, "R_{42}"] &&
    \mathrm{II}^{7} \arrow[dr,"R_{42}"] \\
    & \mathrm{II}^{7} \arrow[rr,pos=0.3,"R_{42}"] &&
    \mathrm{IV}^{7} \\
    \mathrm{I}^{7} \arrow[rr,"T^{7}\circ R_{42}"] \arrow[ur] \arrow[uu,"T^{5}"] && \mathrm{III}^{7} \arrow[ur] \arrow[uu,pos=0.2,"T^{10}"]
    \end{tikzcd}
\end{equation}
\end{theorem}

\begin{proof}
We only have to prove that the square faces of the diagram commute. This follows at once from the involutivity of $R_{42}$, the fact that $T^{5}$ commutes with $R_{42}$
\begin{align*}
R_{42}\circ T^{5}([\underline{x},x+4,x+7,x+11]) &= R_{42}([\underline{x+5},x+9,x,x+4])\\
&=[x+5,x+9,x,\underline{x+2}]\\
&=T^{5}([x,x+4,x+7,\underline{x+9}])\\
&=T^{5}\circ R_{42}([\underline{x},x+4,x+7,x+11]),
\end{align*}
and associativity, for
\begin{align*}
 (R_{42}\circ T^{10})\circ (T^{7}\circ R_{42}) &= R_{42}\circ (T^{10}\circ T^{7})\circ R_{42}\\
 &=R_{42}\circ T^{5}\circ R_{42} = T^{5}\circ R_{42}\circ R_{42} = T_{5}.
\end{align*}
and
\[
T^{10}\circ (T^{7}\circ R_{42}) = (T^{10}\circ T^{7})\circ R_{42} = T^{5}\circ R_{42}.\qedhere
\]
\end{proof}

Now we see the six cadential sets organize as follows.

\begin{definition}[The ABC Conglomerate]
We name the three components of a conglomerate defined by the cadences after jazz standards that exemplify them:
\begin{enumerate}
\item[\textbf{A}] (Alice): $J_3 \leftrightarrow J_4$, connected by $R_{42}$.
\item[\textbf{B}] (Blues): $J_1 \leftrightarrow J_2$, connected by $R_{42}$.
\item[\textbf{C}] (Cherokee): $J_5 \leftrightarrow J_6$, connected by $L_{13}$.
\end{enumerate}
\end{definition}

\begin{remark}
Observe that the $R_{42}$-edges connect the non-common members of these cadential sets inside the commutative diagram \eqref{E:Prism}, while the transpositions $T^{5}$ and $T^{10}$ and the appropriate compositions supply the remaining sides of the corresponding triangles.
\end{remark}

\begin{figure}[h]
\centering
\begin{tikzpicture}[scale=1.5]
% Alice triangle
\node (J3) at (0, 2.5) {$J_3 = \{\mathrm{II}^7, \mathrm{III}^7\}$};
\node (J4) at (3, 2.5) {$J_4 = \{\mathrm{III}^7, \mathrm{IV}^7\}$};
\draw[<->] (J3) -- node[above] {$R_{42}$} (J4);
\node at (0, 3) {\textbf{A (Alice)}};

% Blues triangle
\node (J1) at (0, 1.2) {$J_1 = \{\mathrm{I}^7, \mathrm{II}^7\}$};
\node (J2) at (3, 1.2) {$J_2 = \{\mathrm{I}^7, \mathrm{IV}^7\}$};
\draw[<->] (J1) -- node[above] {$R_{42}$} (J2);
\node at (0, 0.75) {\textbf{B (Blues)}};

% Cherokee edge (bottom)
\node (J5) at (0, 0) {$J_5 = \{\mathrm{V}^7\}$};
\node (J6) at (3, 0) {$J_6 = \{\mathrm{VII}^7\}$};
\draw[<->] (J5) -- node[above] {$L_{13}$} (J6);
\node at (0, -0.5) {\textbf{C (Cherokee)}};

% Vertical connections (dashed to indicate prism structure)
\draw[dashed] (J1) -- (J3);
\draw[dashed] (J2) -- (J4);
\end{tikzpicture}
\caption{The ABC conglomerate structure of tetradic cadential sets. The dashed lines summarize the morphisms from \eqref{E:Prism}.}
\label{fig:conglomerate}
\end{figure}
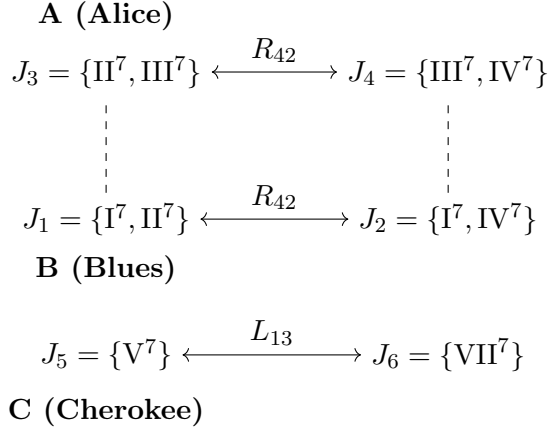

\begin{remark}[Phenomenological interpretation]
The regions have distinct musical characters:
\begin{itemize}
\item \textbf{B} (containing $\mathrm{I}^7$): The ``blues home'' region where tonality is firmly established.
\item \textbf{A} (containing $\mathrm{II}^7$, $\mathrm{III}^7$ but not $\mathrm{I}^7$): Relatively unstable region.
\item \textbf{C} ($\mathrm{V}^7 \leftrightarrow \mathrm{VII}^7$): Tension region, furthest from tonic stability.
\end{itemize}
\end{remark}

\section{Analysis: \textit{Blues for Alice}}

Charlie Parker's \textit{Blues for Alice} (1951) \citep{RealBook} manifests the prism structure explicitly. The chord progression is:
\begin{gather*}
\boxed{\text{F6}} \mid \text{Em7 A7$\flat$9} \mid \text{Dm7} \boxed{\text{G7}} \mid \\
\text{Cm7} \boxed{\text{F7}} \mid \boxed{\text{B$\flat$7}} \mid \boxed{\text{B$\flat$m7}} \text{E$\flat$7} \mid \\
\boxed{\text{Am7}} \text{D7} \mid \text{A$\flat$m7 D$\flat$7} \mid \boxed{\text{Gm7}} \mid \\
\text{C7} \mid \boxed{\text{Am7}} \mid \text{Dm7} \mid \boxed{\text{Gm7}} \text{C7} \mid
\end{gather*}

Before we proceed to the harmonic analysis, it is important to remark that the tetradic interpretation
and cadence model entail a rigidity that, nevertheless, allows us to understand the overall structure
of the tonality phenomenon. Yet, when it comes to analysis and performance, and particularly in jazz, there
is certain flexibility in the interpretation of the chords. Here, the piece opens with F6, but it is clearly
signaling the departing tonic chord. As Mark Levine writes, for instance: ``A C major chord can be notated
as Cmaj7, CM7, C6, C$\substack{9\\6}$, or C$\Delta$, and they all mean pretty much the same thing'' \citep[p. ix]{mL95}.
Teriete \citep[p. 278]{pT20} states, which is quite to the point in our case: ``›Lead Sheets‹ sind so spezifisch wie nötig, aber so abstrakt wie möglich.''\footnote{Lead sheets are as specific as necessary, but as abstract as possible.}

For the purposes of the following analysis, ``articulates'' means that the lead-sheet harmony projects the corresponding root-degree skeleton of the cadential set, possibly through standard jazz procedures such as sixth chords, dominantization, modal mixture, or local ii–V syntax; literal diatonic tetrads will be identified when they occur.

Now we can see that the harmonic structure of \textit{Blues for Alice} traces a path through the BA prism:
\begin{enumerate}
\item Measures 1--3 articulate $J_1 = \{\mathrm{I}^7, \mathrm{II}^7\}$ with continuous melody; F6 realizes the tonic
member, while the G-root sonority in the local ii-V chain projects the II degree in altered dominant form.
\item Measures 4-5 articulate $J_2 = \{\mathrm{I}^7, \mathrm{IV}^7\}$ while the first melodic silence occurs; F7 and
B$\flat$7 project the skeletons of I and IV degrees, dominantized by blues syntax.
\item Measures 6-7 supply B$\flat$m7 as an altered IV degree and Am7 gives the literal III$^{7}$, articulating $J_{4}=\{\text{III}^{7},\text{IV}^{7}\}$.
\item Measures 9-12 display $J_{3}=\{\text{II}^{7},\text{III}^{7}\}$ literally through Gm7 and Am7, embedded in the final turnaround.
\end{enumerate}

The passage through the prism is not equally transparent at every stage. In the opening measures, $J_{1}$ and $J_{2}$ are
articulated via lead-sheet realization, blues dominantization, and local harmonic syntax rather than by bare exhibition
of the diatonic tetrads. The trajectory becomes progressively clearer as the piece approaches the Alice side of the
prism: $J_{4}$ is articulated with an explicit III$^{7}$ and an altered sonority on degree IV, and $J_{3}$ is finally literal with respect to the occurrence of its two members.

\begin{remark}[Phenomenological observation]
In measures 1--3, the melody is comparatively continuous. After measure 3, it becomes more angular and discontinuous, producing an affective contrast roughly analogous to a change of thematic character.

The perceived change of melodic character is reflected in the intervallic profile of the melody. Encoding pitches by semitone distance from middle C, ignoring repetitions and splitting the head at the first rest, the mean absolute melodic interval increases from $3.2500$ to $4.7143$, while the root-mean-square interval increases from $3.5824$ to $5.4423$. Thus the continuation after the first silence is measurably less conjunct and more intervallically mobile than the opening segment.

Thus the title motivates the mnemonic names $B$ (Blues) and $A$ (Alice): the analysis reads the tune as a passage from the blues-home region toward the Alice region.
\end{remark}

\section{Analysis: Cherokee and Quantized Modulation}

Ray Noble's \textit{Cherokee} (1938) \citep{RealBook} demonstrates two types of modulation.

\subsection{Section A: Quantized modulations}

The very opening measures feature a modulation from B$\flat$ to E$\flat$:
\[
\text{B$\flat$maj7} \to \text{F}^{+}7 \to \text{Fm7} \to \text{B$\flat$7} \to \text{E$\flat$maj7} \to \text{A$\flat$9} \to \text{B$\flat$6}
\]

The tonality of B$\flat$ is established with cadential set $J_5 = \{\mathrm{V}^7\}$ (represented by F$^+$7)\footnote{Strictly speaking F$^{+}$7 is not V$^{7}$, but an altered dominant realization of it: it preserves the root, third, and seventh of F7, replacing the fifth by an augmented fifth.}. The modulation to E$\flat$ uses pivot chords $\mathrm{II}^7$ (Fm7) and $\mathrm{VII}^7$; here $\mathrm{VII}_{\text{E}\flat}^{7}=\text{Dm}7\flat 5=\{\text{D},\text{F},\text{A}\flat,\text{C}\}$, which is a subset\footnote{It is important to stress that we do not use arbitrary diatonic supersets, for it would render the very concept of cadential set moot. A cadential member may be realized by inclusion only when the containing set is generated by a local harmonic complex: typically one lead-sheet chord, or two adjacent chords forming a recognized unit. Thus inclusion is constrained by locality and diagnostic minimality.} of A$\flat$9$\cup$B$\flat$6$=\{\text{A}\flat,\text{C},\text{E}\flat,\text{G}\flat,\text{B}\flat,\text{D},\text{F},\text{G}\}$. This is \emph{quantized} because full cadential sets establish each tonality with the corresponding pivots \citep{AgustinAquino2020}.

The modulation back from E$\flat$ to B$\flat$ is as follows:
\[
\text{Dm}7 \to \text{C}7 \to \text{Cm}7 \to \text{Dm}7 \to \text{G}7\flat 9 \to \text{Cm}7 \to \text{F}^{+}7 \to \text{B$\flat$maj}7.
\]

Here the pivots are $\mathrm{III}^{7}$ (Dm7) and $\mathrm{V}^{7}$ (represented by F$^+$7); see \citep[Table 1]{AgustinAquino2020}. It is remarkable that these correspond to the only quantized modulations available for two tonalities that are separated by a fifth (the only other case with a similar property corresponds to tonalities separated by one major third).

\subsection{Bridge: Non-quantized modulations via $P_{42}$}

Here we have the famous bridge
\begin{multline*}
\text{C$\sharp$m7} \to \text{F$\sharp$7} \to \text{Bmaj7} \mid \\
\text{Bm7} \to \text{E7} \to \text{Amaj7} \mid \\
\text{Am7} \to \text{D7} \to \text{Gmaj7} \mid \cdots.
\end{multline*}

These modulations are \emph{not} quantized, but the $P_{42}$ morphism serves as a bridge
\[
P_{42}(\mathrm{I}^7_B) = \text{Bm7} = \mathrm{II}^7_A,
\]
for the set $\{\mathrm{II}^{7},\mathrm{V}^{7}\}$ is cadential, but not minimal (i.e., $J_{5}=\{\mathrm{V}^{7}\}$ suffices). Here the subscripts A and B denote the corresponding keys, not the regions of the conglomerate.

\begin{proposition}
The parallel transformation $P_{42}$ converts the tonic of one key $K$ into the supertonic of the key $K-2$ a whole step below
\[
P_{42}: \mathrm{I}^7_K \mapsto \mathrm{II}^7_{K-2}
\]
which enables rapid whole-step descent established through the cadential set
\[
 \{\mathrm{II}^{7}_{K-2},\mathrm{V}^{7}_{K-2}\}
\]
or the cadence
\[
 \{\mathrm{I}^{7}_{K-2},\mathrm{II}^{7}_{K-2}\}.
\]
\end{proposition}

\begin{remark} The $P_{42}$ morphism \emph{compensates for the absence of quantization}, providing a bridge that mediates in the display of cadential sets. Furthermore, it would be possible to make a \emph{faster} \textit{Cherokee}-like modulation
\[
\text{C$\sharp$m7} \to \text{Bmaj7} \mid \\
\text{Bm7} \to \text{Amaj7} \mid \\
\text{Am7} \to \text{Gmaj7} \mid \cdots
\]
which could be seen as a \emph{descending} chaining. In this schematic variant the dominant harmonies are suppressed; the successive keys are inferred from the displayed $J_{1}$-cadences.
\end{remark}

\section{Comparison with Syntactic Approaches}

Pachet \citep{Pachet1997Solar} analyzed \textit{Blues for Alice} using a syntactic/AI approach with ontological ``shapes''. His system produces parse trees. In a subsequent work, Rousseaux and Pachet \citep{RosseauxPachet1998} acknowledge a fundamental limitation: they cannot capture the \textit{vécu} (lived experience) of transitioning between \textit{régions de complexité}.

Pachet’s analysis is useful precisely because it is successful at the syntactic level. His grammar parses a chord sequence by means of hierarchical objects such as TwoFive, ResolvingSeventh, ChordSubstitution, and BluesShape. Thus it explains how local harmonic objects are combined into a larger blues form. The present approach asks a complementary question: once such syntactic objects have been recognized, which cadential regions are being traversed, and which morphisms relate them? The contrast is summarized in Table 1.

\begin{table}[h]
\centering
\begin{tabular}{ll}
\hline
\textbf{Syntactic Approach} & \textbf{Categorical Approach} \\
\hline
Detects patterns & Identifies transformations \\
Syntagmatic (sequential) & Paradigmatic (sets) \\
TwoFive, Turnaround & $J_1, J_2, J_3, \ldots$ \\
ChordSubstitution (procedure) & $R_{42}$ (morphism) \\
Tree of shapes & Conglomerate with morphisms \\
\hline
\end{tabular}
\caption{Comparison of syntactic and categorical approaches.}
\label{tab:comparison}
\end{table}

% With the transformational framework the $R_{42}$ transformation can be seen as the \textit{vécu} Rousseaux and Pachet seek. Musicians think ``I'm going to the relative'': this is the morphism, not a trace of something deeper. The categorical formalism is more phenomenologically faithful than computational approaches claiming to capture experience.

This comparison should not be read as a rejection of Pachet’s syntactic analysis, but as an overlay of two perspectives. The syntactic layer describes sequential construction; the categorical layer describes paradigmatic relations between cadential sets. In this sense, the morphism $R_{42}$ offers a formal correlate of the \textit{vécu} of the transition between regions of complexity discussed by Rousseaux and Pachet. Musicians may describe this transition informally as “going to the relative”; the categorical apparatus identifies the corresponding morphism.

\section{Stylistic Signatures in the Conglomerate}

The conglomerate structure illuminates how different jazz styles inhabit different regions of cadential space. Each style can be characterized by which pairs of cadential sets it primarily exploits.

\begin{itemize}
\item \textbf{Standard bebop blues}: Primarily navigates within the Blues triangle ($J_1 \leftrightarrow J_2$).
\item \textbf{Parker's innovations}: Traverses the full prism, including the Alice triangle.
\item \textbf{Cherokee's uniqueness}: Exploits the C edge ($J_5 \leftrightarrow J_6$) for quick quantized modulations.
\end{itemize}

This explains Parker's own description of his breakthrough while practicing \textit{Cherokee} \citep{Giddins1987}:
\begin{quote}
 ``Now I'd been getting bored with the stereotyped changes that were being used at the time, and I kept thinking there's bound to be something else. I could hear it sometimes but I couldn't play it. Well, that night I was working over `Cherokee,' and as I did I found that by using the higher intervals of a chord as a melody line and backing them with appropriately related changes, I could play the thing I'd been hearing. I came alive.''
\end{quote}

In the present model, Parker's description can be read as a discovery of harmonic paths that traverse the conglomerate beyond the conventional bebop-blues region.

\section{Conclusion}

The extension of Mazzola's cadential set theory to tetradic harmony via the PLRQ group reveals a conglomerate that organizes the six cadential sets available. This structure:

\begin{enumerate}
\item \textbf{Unifies} the Mazzola (global tonality) and neo-Riemannian (local voice-leading) traditions through the coslice category construction.

\item \textbf{Explains} the phenomenological difference between cadential sets through categorical morphisms.

\item \textbf{Models} a formal correlate of what syntactic approaches leave implicit: the transformational experience musicians have when navigating harmonic space.

\item \textbf{Illuminates} Parker's bebop innovations as an exploration of the full prism.

\item \textbf{Distinguishes} quantized modulation (full cadential establishment) from non-quantized modulation ($R_{42}$ or $P_{42}$ as direct bridges).
\end{enumerate}

The naming convention: Alice ($J_3 \leftrightarrow J_4$), Blues ($J_1 \leftrightarrow J_2$), Cherokee ($J_5 \leftrightarrow J_6$), anchors the abstract structure in the jazz repertoire that first revealed it.

\appendices
\section{Source Code}

The following code in Octave 9.2.0 (using the \texttt{signal} package) was used to calculate the mean absolute melodic intervals and the root-mean-square intervals of the opening segment (\texttt{blues}) and the continuation
after the first rest (\texttt{alice}).

\begin{verbatim}
pkg load signal
blues=[17 12 9 16 12 9 14 16 11 14 13 10 7 9 5 2 7 9 5 4 3 7 10 14 13];
alice=[5 7 5 12 10 5 8 -2 7 15 13 10 8 12 5 7 9 4 0 2 13 11 3 10];
alice=[alice 8 7 17 14 10 2 9 7 12 10 15 12 9 4 7 14 10 2 9];
mean(abs(diff(blues)))
rms(diff(blues))
mean(abs(diff(alice)))
rms(diff(alice))
\end{verbatim}

\bibliographystyle{tMAM}
\bibliography{blues_for_alice}

\end{document}